\documentclass[12pt]{article}
\usepackage[numbers,sort&compress]{natbib}
\usepackage{amssymb,amsmath,mathrsfs,amsfonts,amsthm}
\usepackage{enumerate}
\usepackage{paralist}
\usepackage{color}
\usepackage{epsfig}
\usepackage{graphicx}
\usepackage{setspace}
\usepackage{appendix}
\newtheorem{thm}{Theorem}[section]
\newtheorem{cor}[thm]{Corollary}
\newtheorem{lem}[thm]{Lemma}
\newtheorem{prop}[thm]{Proposition}
\newtheorem*{con}{Conjucture}

\begin{document}
\def\Xint#1{\mathchoice
  {\XXint\displaystyle\textstyle{#1}}%
  {\XXint\textstyle\scriptstyle{#1}}%
  {\XXint\scriptstyle\scriptscriptstyle{#1}}%
  {\XXint\scriptscriptstyle\scriptscriptstyle{#1}}%
  \!\int}
\def\XXint#1#2#3{{\setbox0=\hbox{$#1{#2#3}{\int}$}
  \vcenter{\hbox{$#2#3$}}\kern-.5\wd0}}
\def\ddashint{\Xint=}
\def\dashint{\Xint-}

\newcommand{\R}{\mathbb{R}}

\def\pd#1#2{\frac{\partial#1}{\partial#2}}
\def\dfrac{\displaystyle\frac}
\let\oldsection\section
\renewcommand\section{\setcounter{equation}{0}\oldsection}
\renewcommand\thesection{\arabic{section}}
\renewcommand\theequation{\thesection.\arabic{equation}}
\newtheorem{conj}{Conjection}
\newtheorem{remark}{Remark}[section]
\allowdisplaybreaks

\title{Global dynamics of competition models with  nonlocal dispersals I: Symmetric kernels
\thanks{The first author is supported by Shanghai Postdoctoral Science Foundation (No. 13R21412600) and Postdoctoral Science Foundation of China (No. 2014M551359).
The second author is  supported by Chinese NSF (No. 11201148), Shanghai Pujiang Program (No. 13PJ1402400).} }

\author{Xueli Bai{\thanks{
E-mail: mybxl110@163.com}}, Fang Li{\thanks{Corresponding author.
E-mail: fli@cpde.ecnu.edu.cn}}\\ {Center for PDE, East China Normal
University,}\\{\small 500 Dongchuan Road, Minhang 200241,
Shanghai, P. R. China.} }

\date{}
\maketitle{}

\begin{abstract}
In this paper, the global dynamics of two-species Lotka-Volterra competition models with nonlocal dispersals is studied. Under the assumption that dispersal kernels are symmetric, we prove that except for very special situations, local stability of semi-trivial steady states implies global stability, while when both semi-trivial steady states are locally unstable, the  positive steady state exists and is globally stable. Moreover, our results cover the case that competition coefficients are location-dependent and dispersal strategies are mixture of local and nonlocal dispersals.

\end{abstract}



\section{Introduction}

Dispersal is an important feature of the life histories of many organisms and thus has been a central topic in ecology. In 1951, random diffusion was introduced to model dispersal strategies \cite{Skellam} and there are tremendous studies in this direction \cite{CCbook}, \cite{OL}. In recent years, nonlocal dispersal, which describes the movements of organisms between non-adjacent spatial locations, has attracted lots of attentions from both biologists and mathematicians. We refer \cite{CCLR, CCEM, HG,HMMV,   LHDGHLMUM,  LMNC} and references therein for more details.

The purpose of this paper is to understand the role played by spatial heterogeneity and nonlocal dispersals in the ecology of competing species. This encourages  us to  investigate the global dynamics of the following   models
\begin{equation}\label{original}
\begin{cases}
u_t= d \mathcal{K}[u]  +u(m(x)- u- c v) &\textrm{in } \Omega\times[0,\infty),\\
v_t= D \mathcal{P}[v]  +v(M(x)-b u-  v) &\textrm{in } \Omega\times[0,\infty),
\end{cases}
\end{equation}
where $\Omega$ is a smooth bounded domain in $\mathbb R^n$, $n\geq 1$.
In this model, $u(x,t)$, $v(x,t)$ are the population densities of two competing species, $d, D>0$ are their dispersal rates respectively. The functions $m(x)$, $M(x)$ represent their intrinsic growth rates,
$b,\ c>0$ in $\bar\Omega$ are interspecific competition coefficients.
Two types of  nonlocal operators $\mathcal{K}$ and $\mathcal{P}$ will be considered in this paper. For $\phi\in C(\bar\Omega)$, define
\begin{itemize}
\item[\textbf{(N)}] $\mathcal{K}[\phi] = \int_{\Omega}k(x,y)\phi(y)dy- \int_{\Omega}k(y,x) dy \phi(x)$,  $\mathcal{P}[\phi] = \int_{\Omega}p(x,y)\phi(y)dy- \int_{\Omega}p(y,x) dy \phi(x)$,
\item[\textbf{(D)}] $\mathcal{K}[\phi] = \int_{\Omega}k(x,y)\phi(y)dy-  \phi(x)$,  $\mathcal{P}[\phi] = \int_{\Omega}p(x,y)\phi(y)dy- \phi(x) $,
\end{itemize}
where the kernels $k(x, y)$, $p(x,y)$ describe the rate at which organisms move from point $y$ to point $x$. Types \textbf{(N)} and \textbf{(D)} correspond to no flux boundary condition and lethal boundary condition respectively with local dispersal. See \cite{HMMV} for the details.

Throughout this paper, unless designated otherwise, we assume that
\begin{itemize}
\item[\textbf{(C1)}]  $ m(x), M(x)\in C(\bar{\Omega})$  are nonconstant.
\item[\textbf{(C2)}]  $k(x,y)$, $p(x,y)\in C(\mathbb R^n\times \mathbb R^n)$ are nonnegative and  $k(x,x), \  p(x,x)>0$ in $\mathbb R^n$. Moreover,  $\int_{\mathbb R^n} k(x,y)dy= \int_{\mathbb R^n} k(y,x)dy=1$ and $\int_{\mathbb R^n} p(x,y)dy= \int_{\mathbb R^n} p(y,x)dy=1$.
\item[\textbf{(C3)}] $k(x,y)$, $p(x,y)\in C(\mathbb R^n\times \mathbb R^n)$ are symmetric, i.e., $k(x,y)=k(y,x)$, $p(x,y)=p(y,x)$.
\end{itemize}

To understand the motivation of our studies, two  observations related to system (\ref{original}) are worth mentioning.
First, if the function $m=M$ were a constant, when $0<b,  c<1$,
(\ref{original}) has a unique positive
constant steady state for any diffusion rates $d$ and $D$, which is  globally stable among all positive
continuous initial data.
Secondly, when $0<b,  c<1$ and $m=M$, without the diffusion terms, i.e., $d=D=0$, (\ref{original}) becomes a system of two ordinary differential
equations, whose solutions converge to
$$
\left( {1-b\over 1-bc}m_+(x), {1-c\over 1-bc}m_+(x)\right)\ \ \textrm{for every}\ x\in\Omega,
$$
where $m_+(x)=\max \{ m(x), 0\}$, among all positive
continuous initial data.

Moreover, it is well known interactions between random diffusion and spatial heterogeneity could create some very different  phenomena in
population dynamics. See \cite{HeNi}, \cite{LamNi}, \cite{Lou} and the references therein. However, till now the studies for the corresponding nonlocal models are quite limited. See \cite{BaiLi2015}, \cite{HNShen2012} and the references therein.

Not only in this paper the global dynamics of system (\ref{original}) is investigated, but also the approaches can be employed to study more general systems with either location-dependent coefficients or mixed dispersal strategies.

To better demonstrate our main results and techniques, some explanations are in place. Let $(U(x),V(x))$ denote a nonnegative steady state of (\ref{original}),  then there are {\it at most} three possibilities:
\begin{itemize}
\item $(U,V)= (0,0)$ is called a {\it trivial steady state};
\item $(U,V)=(u_d, 0)$ or $(U,V)=(0, v_D)$ is called a {\it semi-trivial steady state}, where $u_d$, $v_D$ respectively are the positive solutions to single-species models
    \begin{equation}\label{singled}
    d \mathcal{K}[u]  +u(m(x)- u)=0,
    \end{equation}
    and
    \begin{equation}\label{singleD}
     D \mathcal{P}[v]  +v(M(x)- v)=0
    \end{equation}
\item $U>0,\ V>0$, and we call $(U,V)$ a {\it coexistence/positive steady state}.
\end{itemize}

First, for the existence of semi-trivial steady states to competition models, the existence of  positive solutions to single-species models need be studied. In fact,  this issue is of independent interests and has been studied in \cite{BZh} for symmetric operators and in \cite{Coville2010} for nonsymmetric operators of special type. For the convenience of readers, a general result and its detailed proof concerning nonlocal operators satisfying \textbf{(C2)} is included in Appendix A.

From the viewpoint of biology and mathematics, more interesting and challenging issue is to describe the global dynamics of (\ref{original}) based on its local stability when (\ref{original}) admits two semi-trivial steady states. This means two species can both survive without competitors in certain environment, then when they compete for survivals, their dispersal strategies, competition abilities, spatial heterogeneity will  interact with each other and different outcomes might happen. Therefore, our main results will be restricted to this situation, i.e., both semi-trivial steady states exist. To be more precise, we repeat the conjecture proposed in \cite{BaiLi2015}:
\begin{con}
The locally stable steady state  is globally asymptotically stable.
\end{con}
Roughly speaking, to verify this conjecture could be regarded as our intention. Indeed, this conjecture is initially proposed in \cite{Lou} for models with random diffusion and recently has been completely resolved by \cite{HeNi}  provided that $0<bc\leq 1$. It is worth pointing out that in \cite{HeNi}, an intrinsic relations among positive steady states and principal eigenfunctions of linearized problems are discovered, which motivates our work in this paper.
However, due to the lack of regularity in nonlocal models, we have difficulties determining the local stability  by linearized analysis, since principal eigenvalue might not exist. For single-species models or semi-trivial steady states of competition models, it is known that this issue can be resolved by perturbation arguments and spectral analysis. See \cite{BZh}, \cite{HMMV} and so on. Unfortunately, as far as we are concerned, spectral analysis of linearized problem at positive steady states seems too difficult and there are no progress in this direction.  Therefore, instead of verifying this conjecture,   {\it we aim to classify the global dynamics of competition systems based on local stability of semi-trivial steady states.}

For competition models with {\it local dispersals}, it is known that to show global dynamics,  it suffices to demonstrate that every positive steady state is locally stable. See \cite{HsuSW1998} and references therein, where the compactness of solutions orbits is a necessary condition. However, according to previous discussion, not only this property might not hold but also  we need avoid  analyzing local stability of  positive steady state due to its difficulty. Notice that two-species competition models have quite clear solution structure. To be more specific, if one semi-trivial steady state is locally stable while the other is locally unstable, and there is no positive steady state, then the stable one will be globally stable. If two semi-trivial steady states are both locally unstable, then there exists at least one stable positive steady state and moreover the uniqueness will  imply global stability. Thus, in order to handle models with {\it nonlocal dispersals}, we turn our attention back to the well-known solution structures and verify either the nonexistence  or uniqueness of positive steady state directly based on characteristics of nonlocal operators and arguments by contradiction.

Simply for completeness, the global dynamics of system (\ref{original}) with at most one semi-trivial steady state is discussed in Appendix B.
According to Theorem \ref{thm-nonsymmetric}, under this circumstance, the global dynamics of system (\ref{original}) can be characterized completely for both symmetric and nonsymmetric kernels and for all $b,c>0$. Biologically, this result is clearly reasonable. If a species cannot survive without competitors, it will have no chance to survive when another competitor is introduced, regardless of their dispersal strategies, competition abilities, spatial heterogeneity and so on.

{\it The first main result} in this paper gives a complete classification of  the global dynamics to the competition system (\ref{original})  provided that $0<bc\leq 1$.
\begin{thm}\label{thm-main}
Assume that \textbf{(C1)}, \textbf{(C2)}, \textbf{(C3)} hold and $0<bc\leq 1$. Also assume that (\ref{original}) admits two semi-trivial steady states $(u_d, 0)$ and $(0,v_D)$. Then for the global dynamics of the system (\ref{original}) with nonlocal operators $\mathcal{K}$ and $\mathcal{P}$ defined as either type \textbf{(N)} or \textbf{(D)}, there exist exactly four  cases:
\begin{itemize}
\item[(i)] If both $(u_d, 0)$ and $(0,v_D)$ are locally unstable, then the system (\ref{original}) admits a unique positive steady state, which is globally asymptotically stable;
\item[(ii)] If $(u_d, 0)$ is locally unstable and $(0,v_D)$ is locally stable or neutrally stable, then $(0,v_D)$ is globally asymptotically stable;
\item[(iii)] If $(u_d, 0)$ is locally stable or neutrally stable and $(0,v_D)$ is locally unstable, then $(u_d, 0)$ is globally asymptotically stable;
\item[(iv)]  If both $(u_d, 0)$ and $(0,v_D)$ are locally stable or neutrally stable, then $bc=1$, $bu_d = v_D$ and  system (\ref{original}) has a continuum of steady states $\{(su_d, (1-s)v_D),\ 0\leq s\leq 1\}$.
\end{itemize}
\end{thm}

Equipped with the techniques developed in the study of system (\ref{original}), we further investigate the competition system with location-dependent competition coefficients and self-regulations
\begin{equation}\label{general-bc}
\begin{cases}
u_t= d \mathcal{K}[u]  +u(m(x)-b_1(x)u- c(x)v) &\textrm{in } \Omega\times[0,\infty),\\
v_t= D \mathcal{P}[v]  +v(M(x)-b(x)u- c_2(x)v) &\textrm{in } \Omega\times[0,\infty),\\
\end{cases}
\end{equation}
where $b_1,\ c_2$ represent self-regulations. Assume that
\begin{itemize}
\item[\textbf{(C4)}]  $b(x), c(x), b_1(x), c_2(x) \in C(\bar{\Omega})$.
\end{itemize}

{\it The second main result} in this paper completely classifies the global dynamics of system (\ref{general-bc}) provided that
\begin{equation}\label{general-condition}
\max_{\bar\Omega} b(x) \cdot \max_{\bar\Omega} c (x)\leq \min_{\bar\Omega} b_1(x) \cdot \min_{\bar\Omega} c_2 (x).
\end{equation}
\begin{thm}\label{thm-main-locationdependent}
Assume that \textbf{(C1)}, \textbf{(C2)}, \textbf{(C3)}, \textbf{(C4)} hold and (\ref{general-condition}) is valid. Also assume that (\ref{general-bc}) admits two semi-trivial steady states $(u^*_d, 0)$ and $(0,v^*_D)$. Then  the global dynamics of the system (\ref{general-bc}) with nonlocal operators $\mathcal{K}$ and $\mathcal{P}$ defined as either type \textbf{(N)} or \textbf{(D)} can be classified as follows
\begin{itemize}
\item[(i)] If both $(u^*_d, 0)$ and $(0,v^*_D)$ are locally unstable, then the system (\ref{general-bc}) admits a unique positive steady state, which is globally asymptotically stable;
\item[(ii)] If $(u^*_d, 0)$ is locally unstable and $(0,v^*_D)$ is locally stable or neutrally stable, then $(0,v^*_D)$ is globally asymptotically stable;
\item[(iii)] If $(u^*_d, 0)$ is locally stable or neutrally stable and $(0,v^*_D)$ is locally unstable, then $(u^*_d, 0)$ is globally asymptotically stable;
\item[(iv)]  If both $(u^*_d, 0)$ and $(0,v^*_D)$ are locally stable or neutrally stable, then $b(x), c(x), b_1(x), c_2(x)$ must be constants, $bc=b_1c_2$ and $bu^*_d = c_2v^*_D$. Moreover, the system (\ref{general-bc}) has a continuum of steady states $\{(su^*_d, (1-s)v^*_D),\ 0\leq s\leq 1\}$.
\end{itemize}
\end{thm}

This greatly extend results in \cite{HNShen2012}, where for certain cases,  coexistence/extinction phenomena is studied.

{\it The last main result} in this paper is about the competition system with mixed dispersal strategies
\begin{equation}\label{general-mix}
\begin{cases}
u_t= d \hat{\mathcal{K}}_{\alpha}[u]  +u(m(x)-b_1(x)u- c(x)v) &\textrm{in } \Omega\times[0,\infty),\\
v_t= D \hat{\mathcal{P}}_{\beta}[v]  +v(M(x)-b(x)u- c_2(x)v) &\textrm{in } \Omega\times[0,\infty),\\
(1-\alpha)\partial u/\partial \nu =(1-\beta) \partial v/\partial \nu=0   &\textrm{on } \partial\Omega,
\end{cases}
\end{equation}
where $\nu$ denotes the unit outer normal vector on $\partial\Omega$, for $\phi\in C(\bar\Omega)$
$$
\hat{\mathcal{K}}_{\alpha}[\phi] = \alpha \left\{ \int_{\Omega}k(x,y)\phi(y)dy- \int_{\Omega}k(y,x) dy \phi(x)  \right\}+(1-\alpha) \Delta \phi,
$$
$$
\hat{\mathcal{P}}_{\beta}[\phi] = \beta \left\{ \int_{\Omega}p(x,y)\phi(y)dy- \int_{\Omega}p(y,x) dy \phi(x)  \right\}+(1-\beta) \Delta \phi,
$$
and $0\leq \alpha,\beta\leq 1$. The global dynamics of this system can be classified when (\ref{general-condition}) holds.
\begin{thm}\label{thm-main-mix}
Assume that \textbf{(C1)}, \textbf{(C2)}, \textbf{(C3)}, \textbf{(C4)} hold and (\ref{general-condition}) is valid. Also assume that (\ref{general-mix}) admits two semi-trivial steady states $(\hat{u}_d, 0)$ and $(0,\hat{v}_D)$. Then  there exist exactly four  cases:
\begin{itemize}
\item[(i)] If both $(\hat{u}_d, 0)$ and $(0,\hat{v}_D)$ are locally unstable, then the system (\ref{general-mix}) admits a unique positive steady state, which is globally asymptotically stable;
\item[(ii)] If $(\hat{u}_d, 0)$ is locally unstable and $(0,\hat{v}_D)$ is locally stable or neutrally stable, then $(0,\hat{v}_D)$ is globally asymptotically stable;
\item[(iii)] If $(\hat{u}_d, 0)$ is locally stable or neutrally stable and $(0,\hat{v}_D)$ is locally unstable, then $(\hat{u}_d, 0)$ is globally asymptotically stable;
\item[(iv)]  If both $(\hat{u}_d, 0)$ and $(0,\hat{v}_D)$ are locally stable or neutrally stable, then $b(x), c(x), b_1(x), c_2(x)$ must be constants, $bc=b_1c_2$, $b\hat{u}_d = c_2\hat{v}_D$ and the system (\ref{general-mix}) has a continuum of steady states $\{(s\hat{u}_d, (1-s)\hat{v}_D),\ 0\leq s\leq 1\}$.
\end{itemize}
\end{thm}

This paper is organized as follows. Section 2 provides some background properties and  a general result concerning global dynamics of two-species competition models, regardless of whether the dispersal kernels are symmetric or not. Section 3 is devoted to the proof of Theorem \ref{thm-main} and  proofs of Theorems \ref{thm-main-locationdependent} and \ref{thm-main-mix} are included in Section 4.

\section{Preliminaries}
This section is devoted to the describing the scheme in establishing the main results. It is worth pointing out that throughout this section, assumption \textbf{(C3)} is not imposed, in other words, the nonlocal operators can be {\it nonsymmetric}.

From now on, for simplicity, we combine two types of nonlocal operators $\mathcal{K}$ and $\mathcal{P}$ as follows
\begin{equation}\label{K-kernel}
\mathcal{K}[u] = \int_{\Omega}k(x,y)u(y)dy- a_d(x)  u(x),
\end{equation}
\begin{equation}\label{P-kernel}
\mathcal{P}[u] = \int_{\Omega}p(x,y)u(y)dy- a_D(x)   u(x),
\end{equation}
where for type \textbf{(N)}, $a_d(x)=\int_{\Omega}k(y,x) dy,\ a_D(x)=\int_{\Omega}p(y,x) dy$, while for type \textbf{(D)}, $a_d(x)\equiv 1,\ a_D(x)\equiv 1$.

For clarity, we will focus on competition model (\ref{original}) and always assume that
there exist two semi-trivial steady states $(u_d, 0)$ and $(0,v_D)$.

First of all,  the linearized operator of (\ref{original})  at $(u_d, 0)$ is
\begin{equation}\label{lin-d}
\mathcal{L}_{(u_d,0)} {\phi\choose\psi}={d\mathcal{K}[\phi]+[m(x)-2u_d]\phi-cu_d\psi \choose D\mathcal{P}[\psi]+[M(x)-bu_d]\psi}.
\end{equation}
Also,  the linearized operator of (\ref{original})  at $(0, v_D)$ is
\begin{equation}\label{lin-D}
\mathcal{L}_{(0, v_D)} {\phi\choose\psi}={ d \mathcal{K}[\phi] +[m (x)-cv_D]\phi  \choose D\mathcal{P}[\psi] +[M(x)-2v_D]\psi-bv_D\phi }.
\end{equation}
Denote
\begin{eqnarray}\label{PEV}
&&  \mu_{(u_d,0)}=\sup \left\{\textrm{Re}\, \lambda\, |\, \lambda\in \sigma(D\mathcal{P}+[M(x)-bu_d])   \right\} \\
&&  \nu_{(0, v_D)}= \sup \left\{\textrm{Re}\, \lambda\, |\, \lambda\in \sigma( d \mathcal{K}  +[m (x)-cv_D])   \right\}.\nonumber
\end{eqnarray}
It is known that the signs of $\mu_{(u_d,0)}$ and $\nu_{(0, v_D)}$ determine the local stability/instability of
$(u_d,0)$ and  $(0, v_D)$ respectively.  This is explicitly stated as follows and the proof is omitted since it is standard.
\begin{lem}\label{lm-signs}
Assume that the assumptions \textbf{(C1)}, \textbf{(C2)} hold. Then
\begin{itemize}
\item[(i)] $(u_d,0)$ is locally unstable if $\mu_{(u_d,0)}>0$; $(u_d,0)$ is locally stable if $\mu_{(u_d,0)}<0$; $(u_d,0)$ is neutrally stable if $\mu_{(u_d,0)}=0$.
\item[(ii)] $(0, v_D)$ is locally unstable if $\nu_{(0, v_D)}>0$; $(0, v_D)$ is locally stable if $\nu_{(0, v_D)}<0$; $(0, v_D)$ is neutrally stable if $\nu_{(0, v_D)}=0$.
\end{itemize}
\end{lem}

Remark that as explained in the introduction, due to lack of regularity, in general $\mu_{(u_d,0)}$ and $\nu_{(0, v_D)}$ might not be principal eigenvalues of the corresponding linearized operators. Thus the characterization of $\mu_{(u_d,0)}$ and $\nu_{(0, v_D)}$ need be justified carefully. See \cite{CovilleLiWang} and its references for more discussions. When nonlocal operators are {\it symmetric}, the variational characterization is useful, i.e. for example
$$
\mu_{(u_d,0)}= \sup_{0 \neq \psi\in L^2} \frac{\int_{\Omega} \left( D\psi \mathcal{P}[\psi]+[M(x)-bu_d]\psi^2 \right) dx }{\int_{\Omega} \psi^2 dx}.
$$
See Appendix A for the case that  nonlocal operators are {\it nonsymmetric}.

The following result explains how to characterize the global dynamics of the competition model (\ref{original}) with two semi-trivial steady states.

\begin{thm}\label{thm-monotone}
Assume that the assumptions \textbf{(C1)}, \textbf{(C2)} hold. Also assume that  system (\ref{original}) admits two semi-trivial steady states $(u_d, 0)$ and $(0,v_D)$. We have the following three possibilities:
\begin{itemize}
\item[(i)]   If both $\mu_{(u_d,0)}$ and $\nu_{(0, v_D)}$,  defined in (\ref{PEV}), are positive,  the system (\ref{original}) at least has one positive steady state in $C(\bar\Omega)\times C(\bar\Omega)$. If in addition, assume that the system (\ref{original}) has a unique positive steady state, denoted by $(u_d, v_D)$, then $(u_d, v_D)$ is globally asymptotically stable.
\item[(ii)]   If $\mu_{(u_d,0)}$ defined in (\ref{PEV}) is positive and no positive steady states of system (\ref{original}) exist, then the semi-trivial steady state $(0,v_D)$ is globally asymptotically stable.
\item[(iii)]   If $\nu_{(0, v_D)}$ defined in (\ref{PEV}) is positive and system (\ref{original}) does not admit positive steady states, then the semi-trivial steady state $(u_d,0)$ is globally asymptotically stable.
\end{itemize}
\end{thm}

This result is a natural extension of the competition model with
random diffusion. Its proof is based on upper/lower solution method and we omit the details here since it is almost the same as that of \cite[Theorem 2.1]{BaiLi2015}, where a simplified nonlocal operator is considered.

\begin{remark}
It is routine to verify that Theorem \ref{thm-monotone} also holds for systems (\ref{general-bc}) and (\ref{general-mix}). Indeed, one sees from the proof of Theorem \ref{thm-monotone} that for models with only nonlocal dispersals, $\mu_{(u_d,0)}$ and $\nu_{(0, v_D)}$ might not be principal eigenvalues, thus the constructions of upper/lower solutions rely on the principal eigenfunctions of suitably perturbed eigenvalue problems which admit principal eigenvalues. However, when local diffusion is incorporated, the existence of principal eigenvalues is always guaranteed, which simply makes the arguments easier.
\end{remark}

\section{Proof of Theorem \ref{thm-main}}
To better demonstrate the proof of Theorem \ref{thm-main}, some properties of local stability and positive steady states of (\ref{original}) will be analyzed first.

The following result is about the classification of local stability.

\begin{prop}\label{prop-localstability}
Assume that \textbf{(C1)}, \textbf{(C2)}, \textbf{(C3)} hold and $0<bc\leq 1$. Then there exist exactly four alternatives as follows.
\begin{itemize}
\item[(i)] $\mu_{(u_d,0)}>0$,  $\nu_{(0, v_D)}>0$;
\item[(ii)] $\mu_{(u_d,0)}>0$,  $\nu_{(0, v_D)}\leq 0$;
\item[(iii)] $\mu_{(u_d,0)}\leq 0$,  $\nu_{(0, v_D)}>0$;
\item[(iv)] $\mu_{(u_d,0)}= \nu_{(0, v_D)}=0$.
\end{itemize}
Moreover,  $(iv)$  holds if and only if $bc=1$ and $bu_d= v_D$.
\end{prop}

\begin{proof}
Only need show that when $\mu_{(u_d,0)}\leq 0$,  $\nu_{(0, v_D)}\leq 0$, we have $\mu_{(u_d,0)}= \nu_{(0, v_D)}=0$, $b=c=1$ and $u_d= v_D$.

Note that
$$
\mu_{(u_d,0)}= \sup_{0 \neq \psi\in L^2} \frac{\int_{\Omega} \left( D\psi \mathcal{P}[\psi]+[M(x)-bu_d]\psi^2 \right) dx }{\int_{\Omega} \psi^2 dx}\leq 0.
$$
Thus one sees that
$$
\int_{\Omega} \left( Dv_D \mathcal{P}[v_D]+[M(x)-bu_d]v_D^2 \right) dx\leq 0,
$$
and thus due to (\ref{singleD}) it follows that
\begin{equation*}
\int_{\Omega} \left( -[M(x)- v_D]v_D^2+[M(x)-bu_d]v_D^2 \right) dx\leq 0,
\end{equation*}
i.e.,
\begin{equation}\label{pf-localstability1}
\int_{\Omega} \left( v_D^3  -bu_dv_D^2 \right) dx\leq 0.
\end{equation}

Similarly, $\nu_{(0, v_D)}\leq 0$ and (\ref{singled}) give that
$$
\int_{\Omega} \left( du_d \mathcal{K}[u_d]+[m(x)-cv_D]u_d^2 \right) dx\leq 0
$$
and
\begin{equation}\label{pf-localstability2}
\int_{\Omega} \left( u_d^3  -cv_Du_d^2 \right) dx\leq 0.
\end{equation}

Now by multiplying (\ref{pf-localstability2}) by $b^3$ and using the condition $0<bc\leq 1$, we have
$$
\int_{\Omega} \left( (bu_d)^3  - v_D(bu_d)^2 \right) dx\leq \int_{\Omega} \left( (bu_d)^3  - bcv_D(bu_d)^2 \right) dx\leq 0,
$$
which, together with (\ref{pf-localstability1}), implies that
\begin{equation}\label{pf-localstability-key}
\int_{\Omega}(bu_d-v_D)^2(bu_d+v_D)dx\leq 0.
\end{equation}
Therefore, all previous inequalities should be equalities. Hence it is obvious that $\mu_{(u_d,0)}= \nu_{(0, v_D)}=0$, $bc=1$ and $bu_d=v_D$. On the other hand, if $bc=1$ and $bu_d=v_D$, then it is easy to check that $\mu_{(u_d,0)}= \nu_{(0, v_D)}=0$.
\end{proof}

Next, we give the description of number of positive steady states of system (\ref{original}). Simply speaking, this number could be zero, one or infinity. As explained in the introduction after \textbf{Conjecture}, there is essential difference between models with  local and nonlocal dispersals.  Therefore, to study global dynamics of competition systems with nonlocal dispersals,  different from approaches known for local models, numbers of   positive steady states need be classified separately. To be more specific, our arguments rely on exploring characteristics of nonlocal operators, as well as some integral relations inspired by \cite{HeNi}.

\begin{prop}\label{prop-steadystates}
Assume that \textbf{(C1)}, \textbf{(C2)}, \textbf{(C3)} hold and $0<bc\leq 1$. Then one of the following situations must be true.
\begin{itemize}
\item[(i)] (\ref{original}) has no positive steady states;
\item[(ii)] (\ref{original}) admits exactly one positive steady state;
\item[(iii)] (\ref{original}) has infinitely many positive steady states.
\end{itemize}
Moreover,  $(iii)$  happens if and only if $bc=1$, $b u_d= v_D$, and then all the positive steady states of (\ref{original}) consist of $(su_d, (1-s)v_D)$, $0<s<1$.
\end{prop}

\begin{proof}
Suppose that (\ref{original}) admits two different positive steady states $(u,v)$ and $(u^*,v^*)$, w.l.o.g., $u>u^*$, $v<v^*$.
We will show that (iii) happens.

First, set $w= u- u^*>0$ and $z=v-v^*<0$ and it is standard to check that
\begin{equation}\label{pf-equation-w-z}
\begin{cases}
d\mathcal{K}[w] +(m-u-cv)w-u^*w-cu^*z=0,\\
D\mathcal{P}[z] + (M-bu-v) z- bv^*w -v^* z=0.
\end{cases}
\end{equation}
Using the equation satisfied by $u$, one has
$$
d\left(u \mathcal{K}[w] - w\mathcal{K}[u]  \right) =u u^* (w+cz).
$$
This yields that
\begin{equation}\label{pf-nonpositive}
d\int_{\Omega}\left(-u \mathcal{K}[u^*] + u^* \mathcal{K}[u]  \right) \frac{w^2}{u u^* } dx =  \int_{\Omega}(w+cz)w^2 dx.
\end{equation}
We claim that $\int_{\Omega}(w+cz)w^2 dx\leq 0$.

To prove this claim, let us calculate the left hand side of (\ref{pf-nonpositive}). Note that assumption \textbf{(C3)}, i.e. $k(x,y)$ is symmetric, is important in the following computations.
\begin{eqnarray}\label{pf-exchange-xy-1}
&& d\int_{\Omega}\left(-u \mathcal{K}[u^*] + u^* \mathcal{K}[u]  \right) \frac{w^2}{u u^* } dx \cr
&=& d\int_{\Omega}\int_{\Omega} k(x,y)\left[ u^*(x)u(y)-u(x)u^*(y) \right] \frac{(u(x)-u^*(x))^2}{u(x) u^*(x) } dy dx\cr
&=&  d\int_{\Omega}\int_{\Omega} k(x,y)\left[ u^*(x)u(y)-u(x)u^*(y) \right]\left( \frac{u(x)}{u^*(x) }+\frac{u^*(x)}{u(x) } \right) dy dx,
\end{eqnarray}
where $\int_{\Omega}\int_{\Omega} k(x,y)\left[ u^*(x)u(y)-u(x)u^*(y) \right] dydx =0$ is used. By exchanging $x$ and $y$, we have
\begin{eqnarray}\label{pf-exchange-xy-2}
&& d\int_{\Omega}\left(-u \mathcal{K}[u^*] + u^* \mathcal{K}[u]  \right) \frac{w^2}{u u^* } dx \cr
&=&  d\int_{\Omega}\int_{\Omega} k(y,x)\left[ u^*(y)u(x)-u(y)u^*(x) \right]\left( \frac{u(y)}{u^*(y) }+\frac{u^*(y)}{u(y) } \right) dy dx.
\end{eqnarray}
Due to (\ref{pf-exchange-xy-1}) and (\ref{pf-exchange-xy-2}), one sees that
\begin{eqnarray*}
&& d\int_{\Omega}\left(-u \mathcal{K}[u^*] + u^* \mathcal{K}[u]  \right) \frac{w^2}{u u^* } dx \cr
&=&  {d\over 2}\int_{\Omega}\int_{\Omega} k(x,y)\left[ u^*(x)u(y)-u(x)u^*(y) \right]\left( \frac{u(x)}{u^*(x) }+\frac{u^*(x)}{u(x) } -\frac{u(y)}{u^*(y) }-\frac{u^*(y)}{u(y) } \right) dy dx\\
&=&  {d\over 2}\int_{\Omega}\int_{\Omega} k(x,y)\left[ u^*(x)u(y)-u(x)u^*(y) \right]^2\left(  \frac{1}{u(x) u(y) } -\frac{1}{u^*(x)u^*(y) }  \right) dy dx\\
&\leq & 0
\end{eqnarray*}
since $u>u^*$. The claim is proved, i.e., $\int_{\Omega}(w+cz)w^2 dx\leq 0$.

Similarly, using (\ref{pf-equation-w-z}) and the equation satisfied by $v$, we have
$$
D\left(v \mathcal{P}[z] - z\mathcal{P}[v]  \right) =v v^* (bw+z),
$$
which gives that
$$
D\int_{\Omega}\left(-v \mathcal{P}[v^*] + v^*\mathcal{P}[v]  \right) \frac{z^2}{v v^* } dx =  \int_{\Omega}(bw+z)z^2 dx.
$$
Similar to the proof of the previous claim, we obtain
\begin{eqnarray}\label{pf-key}
&& \int_{\Omega}(bw+z)z^2 dx\cr
&=&D\int_{\Omega}\left(-v \mathcal{P}[v^*] + v^*\mathcal{P}[v]  \right) \frac{z^2}{v v^* } dx\cr
&=&  {D\over 2}\int_{\Omega}\int_{\Omega} p(x,y)\left[ v^*(x)v(y)-v(x)v^*(y) \right]^2\left(  \frac{1}{v(x) v(y) } -\frac{1}{v^*(x)v^*(y) }  \right) dy dx\cr
&\geq & 0
\end{eqnarray}
since $v<v^*$.

Now we have derived two important inequalities:
\begin{equation}\label{important}
\int_{\Omega}(w+cz)w^2 dx\leq 0,\ \ \ \int_{\Omega}(bw+z)z^2 dx\geq 0.
\end{equation}
Multiplying the second one by $c^3$ and subtracting the first one, it follows that
\begin{eqnarray}\label{pf-keyrelation}
0 &\leq& \int_{\Omega}(cbw+cz)(cz)^2 dx -\int_{\Omega}(w+cz)w^2 dx\cr
&\leq &  \int_{\Omega}(w+cz)(cz)^2 dx -\int_{\Omega}(w+cz)w^2 dx\cr
&=& \int_{\Omega}(w+cz)^2(cz-w) dx,
\end{eqnarray}
where $bc\leq 1$ is used in the second inequality.
The assumption $w= u- u^*>0$ and $z=v-v^*<0$ indicates that $w+cz =0$ in $\bar\Omega$ and all the previous inequalities should be equalities. Hence we also have $bc=1$ and $bw+z =0$ (i.e., $w+cz =0$) in $\bar\Omega$.

Moreover, note that $w+cz =0$ is equivalent to $u+cv =u^* +cv^*$ and recall the equation satisfied by $u$, $u^*$, it is standard to show that $u^*= \alpha u$, where $0<\alpha<1$. Similarly, it can be verified that $v^*= \beta v$, where $\beta>1$. Then using $u+cv =u^* +cv^*$ again, we have
$$
u=c{\beta-1\over 1-\alpha} v.
$$
Recall that $(u,v)$ satisfies
\begin{equation*}
\begin{cases}
d \mathcal{K}[u]  +u(m(x)-u- cv)=0, \\
D \mathcal{P}[v]  +v(M(x)-bu- v)=0,
\end{cases}
\end{equation*}
which can be written as
\begin{equation*}
\begin{cases}
d \mathcal{K}[c{\beta-1\over 1-\alpha} v]  +c{\beta-1\over 1-\alpha} v(m(x)-c{\beta-1\over 1-\alpha} v- cv)=0, \\
D \mathcal{P}[v]  +v(M(x)-bc{\beta-1\over 1-\alpha} v- v)=0,
\end{cases}
\end{equation*}
which becomes
\begin{equation*}
\begin{cases}
d \mathcal{K}[  v]  +  v(m(x)-c{\beta-\alpha\over 1-\alpha} v )=0, \\
D \mathcal{P}[v]  +v(M(x)- {\beta-\alpha\over 1-\alpha} v )=0.
\end{cases}
\end{equation*}
This, together with Corollary \ref{Cor-single}, shows that
$$
u_d= c{\beta-\alpha\over 1-\alpha} v,\ \ \ v_D = {\beta-\alpha\over 1-\alpha} v.
$$
Therefore, $bc=1$, $b u_d= v_D$ and all the positive steady states of (\ref{original}) consist of $(su_d, (1-s)v_D)$, $0<s<1$.
\end{proof}

Now we complete the proof of Theorem \ref{thm-main}  on the basis of Propositions \ref{prop-localstability} and \ref{prop-steadystates}.

\begin{proof}[Proof of Theorem \ref{thm-main}]
First of all, notice that Propositions \ref{prop-localstability} and  \ref{prop-steadystates} immediately  yield case (iv). Thus for cases (i), (ii), (iii), Propositions \ref{prop-localstability} and \ref{prop-steadystates} indicate that system (\ref{original}) admit {\it either no positive steady states or a unique positive steady state.}

Therefore, for case (i), i.e., $\mu_{(u_d,0)}<0$,  $\nu_{(0, v_D)}<0$, thanks to Theorems \ref{thm-monotone}, one easily sees   that  system (\ref{original}) admits a unique positive steady state, which is globally asymptotically stable.

It remains to prove case (ii), since (iii) can be handled similarly.
According to Theorem \ref{thm-monotone}, to prove that $(0,v_D)$ is globally asymptotically stable,  it suffices to show that (\ref{original}) admits no positive steady states.
    Suppose that (\ref{original}) admits a positive steady state $(u,v)$, i.e., $(u,v)$ satisfies
    \begin{equation*}
     \begin{cases}
    d \mathcal{K}[u]  +u(m(x)-u- cv)=0, \\
    D \mathcal{P}[v]  +v(M(x)-bu- v)=0.
    \end{cases}
    \end{equation*}
    Denote $(u^*, v^*)= (0,v_D)$ and set $w= u- u^*= u>0$, $z=v-v^*<0$. Similar to the computation of (\ref{pf-key}), one has
    \begin{eqnarray*}
 &&  \int_{\Omega}(bu+z)z^2 dx = \int_{\Omega}(bw+z)z^2 dx\cr
&=&  {D\over 2}\int_{\Omega}\int_{\Omega} p(x,y)\left[ v^*(x)v(y)-v(x)v^*(y) \right]^2\left(  \frac{1}{v(x) v(y) } -\frac{1}{v^*(x)v^*(y) }  \right) dy dx \geq  0.
    \end{eqnarray*}
    However,
    \begin{eqnarray*}
    0 &\geq& \nu_{(0, v_D)}= \sup_{0 \neq \phi\in L^2} \frac{\int_{\Omega} \left( d\phi \mathcal{K}[\phi]+[m(x)-cv_D]\phi^2 \right) dx }{\int_{\Omega} \phi^2 dx}\\
    &\geq &  \frac{\int_{\Omega} \left( d u \mathcal{K}[ u ]+[m(x)-cv_D]u^2 \right) dx }{\int_{\Omega} u^2 dx}\\
    &=&\frac{\int_{\Omega} \left( -[m(x)-u-cv]u^2+[m(x)-cv_D]u^2 \right) dx }{\int_{\Omega} u^2 dx}\\
    &=& \frac{\int_{\Omega} (u+cz)u^2   dx }{\int_{\Omega} u^2 dx}.
    \end{eqnarray*}
    Putting together the above two inequalities:
    \begin{equation}\label{important-2}
    \int_{\Omega}(bu+z)z^2 dx\geq 0,\ \ \ \int_{\Omega} (u+cz)u^2   dx \leq 0.
    \end{equation}
    similar to (\ref{pf-keyrelation}), we obtain
    $$
    \int_{\Omega} (u+cz)^2(cz-u)dx\geq 0,
    $$
    where $0<bc\leq 1$ is used. Hence $u+cz =0 $ in $\bar\Omega$ and all the previous inequalities should be equalities. In particular, $bc=1$ and $bu+z =0$. Note that $bu+z =0$ means $bu+v = v_D$. Then based on the equations satisfied by $v$ and $v_D$ respectively, it is routine to show that $v= \alpha v_D$, where $0<\alpha<1$. Thus, $u= c(1-\alpha)v_D$. Then plugging  $v= \alpha v_D$ and $u= c(1-\alpha)v_D$ into the equation satisfied by $u$, we have
    $$
    d c(1-\alpha) \mathcal{K}[v_D]  +c(1-\alpha)v_D (m(x)-c v_D)=0,
    $$
    which indicates that $u_d = c v_D$, i.e., $bu_d= v_D$. This yields a contradiction due to Proposition \ref{prop-localstability}.
\end{proof}

\section{More general competition models}
This section is devoted to the proofs of Theorems \ref{thm-main-locationdependent} and \ref{thm-main-mix}, which are about systems (\ref{general-bc}) and (\ref{general-mix}) respectively. The general approaches are similar to that of Theorem \ref{thm-main}. In fact, the  local stability of semi-trivial steady states and number of positive steady states to  systems (\ref{general-bc}) and (\ref{general-mix}) respectively can be classified similarly as in Propositions \ref{prop-localstability} and \ref{prop-steadystates}. To avoid being redundant, the statements of these results  are omitted and we will simply include all the necessary details in the proofs of
Theorems \ref{thm-main-locationdependent} and \ref{thm-main-mix}.

\begin{proof}[Proof of Theorem \ref{thm-main-locationdependent}]
(i) Assume that both $(u^*_d, 0)$ and $(0,v^*_D)$ are locally unstable, based on Theorem \ref{thm-monotone}  and the remark after it, existence of positive steady states follows immediately and it suffices to show that system (\ref{general-bc}) admits a unique positive steady state.

Similar to the proof of Proposition \ref{prop-steadystates}, suppose that (\ref{general-bc}) admits two different positive steady states $(u,v)$ and $(u^*,v^*)$, w.l.o.g., $u>u^*$, $v<v^*$ and set $w= u- u^*>0$ and $z=v-v^*<0$. Then similar to the computations in deriving (\ref{important}), we have
\begin{equation}\label{important-general-bc-original}
\int_{\Omega}\left(b_1(x)w+c(x)z\right)w^2 dx\leq 0,\ \ \ \int_{\Omega}\left(b(x)w+c_2(x)z\right)z^2 dx\geq 0.
\end{equation}
This further implies that
\begin{equation}\label{important-general-bc}
\int_{\Omega}\left([\min_{\bar\Omega}b_1] w+[\max_{\bar\Omega}c]z\right)w^2 dx\leq 0,\ \ \ \int_{\Omega}\left([\max_{\bar\Omega}b]w+[\min_{\bar\Omega}c_2]z\right)z^2 dx\geq 0.
\end{equation}
According to (\ref{general-condition}), the above inequalities  (\ref{important-general-bc-original}) and (\ref{important-general-bc}), the arguments after (\ref{important}) can be applied, one sees that  $b(x), c(x), b_1(x), c_2(x)$ must be constants, $bc=b_1c_2$ and $bu^*_d = c_2v^*_D$. However, it is standard to check that $bc=b_1c_2$ and $bu^*_d = c_2v^*_D$ imply that both $(u^*_d, 0)$ and $(0,v^*_D)$ are neutrally stable. This is a contradiction.

(ii) Assume that $(u^*_d, 0)$ is locally unstable and $(0,v^*_D)$ is locally stable or neutrally stable, according to Theorem \ref{thm-monotone}  and the remark after it, it suffices to show that there is no positive steady state.

Suppose   that (\ref{general-bc}) has one positive steady state $(u,v)$. Then denote  $(u^*, v^*)= (0,v_D)$ and set $w= u- u^*= u>0$, $z=v-v^*<0$. Similar to the proof of Theorem \ref{thm-main} (ii), one can first derive that
$$
\int_{\Omega}\left(b_1(x)u+c(x)z\right)u^2 dx\leq 0,\ \ \ \int_{\Omega}\left(b(x)u+c_2(x)z\right)z^2 dx\geq 0,
$$
which gives that
$$
\int_{\Omega}\left([\min_{\bar\Omega}b_1] u+[\max_{\bar\Omega}c]z\right)u^2 dx\leq 0,\ \ \ \int_{\Omega}\left([\max_{\bar\Omega}b]u+[\min_{\bar\Omega}c_2]z\right)z^2 dx\geq 0.
$$
These inequalities, together with (\ref{general-condition}), yield that $b(x), c(x), b_1(x), c_2(x)$ must be constants, $bc=b_1c_2$ and $bu^*_d = c_2v^*_D$ by similar arguments after (\ref{important-2}). Therefore, again one can easily check that both $(u^*_d, 0)$ and $(0,v^*_D)$ are neutrally stable, which is a contradiction.

(iii) This case can be handled similarly as case (ii).

(iv) Assume that both $(u^*_d, 0)$ and $(0,v^*_D)$ are locally stable or neutrally stable. Then mainly following the proof of Proposition \ref{prop-localstability}, the result can be derived.
\end{proof}

\begin{proof}[Proof of Theorem \ref{thm-main-mix}]
(i) Assume that both $(\hat{u}_d, 0)$ and $(0,\hat{v}_D)$ are locally unstable. Same as the proof of Theorem \ref{thm-main-locationdependent} (i), it suffices to show that system (\ref{general-mix}) admits a unique positive steady state. Thus suppose that (\ref{general-bc}) admits two different positive steady states $(u,v)$ and $(\hat{u},\hat{v})$, w.l.o.g., $u>\hat{u}$, $v<\hat{v}$ and set $w= u- \hat{u}>0$ and $z=v-\hat{v}<0$.

Since local and nonlocal dispersals are mixed in this situation, we will include more computations in deriving the following inequalities:
\begin{equation}\label{important-3}
\int_{\Omega}\left(b_1(x)w+c(x)z\right)w^2 dx\leq 0,\ \ \ \int_{\Omega}\left(b(x)w+c_2(x)z\right)z^2 dx\geq 0.
\end{equation}
For this purpose, first, similar to (\ref{pf-nonpositive}), it is routine to check that
\begin{equation}\label{pf-3}
d\int_{\Omega}\left(-u \mathcal{K}[\hat{u}] + \hat{u} \mathcal{K}[u]  \right) \frac{w^2}{u \hat{u} } dx =  \int_{\Omega}(b_1(x)w+c(x)z)w^2 dx.
\end{equation}
Then due to \textbf{(C3)},    the left hand side of  (\ref{pf-3}) is calculated as follows
\begin{eqnarray*}
&& d\int_{\Omega}\left(-u \mathcal{K}[\hat{u}] + \hat{u} \mathcal{K}[u]  \right) \frac{w^2}{u \hat{u} } dx \\
&=& d\alpha \int_{\Omega}\int_{\Omega} k(x,y)\left[ \hat{u}(x)u(y)-u(x)\hat{u}(y) \right] \frac{(u(x)-\hat{u}(x))^2}{u(x) \hat{u}(x) } dy dx\\
&& + d(1-\alpha) \int_{\Omega} \left( -u \Delta \hat{u}+  \hat{u} \Delta u\right)\frac{(u(x)-\hat{u}(x))^2}{u(x) \hat{u}(x) } dx\\
&=&  {d\over 2}\int_{\Omega}\int_{\Omega} k(x,y)\left[ \hat{u}(x)u(y)-u(x)\hat{u}(y) \right]^2\left(  \frac{1}{u(x) u(y) } -\frac{1}{\hat{u}(x)\hat{u}(y) }  \right) dy dx\\
&& + d(1-\alpha) \int_{\Omega} | u(x) \nabla  \hat{u}(x) - \hat{u}(x) \nabla u(x)|^2\left(  \frac{1}{u^2(x) } -\frac{1}{\hat{u}^2(x) }  \right)dx\\
&\leq & 0.
\end{eqnarray*}
Thus due to (\ref{pf-3}), we have
$$
\int_{\Omega}\left(b_1(x)w+c(x)z\right)w^2 dx\leq 0,
$$
while the other inequality in (\ref{important-3}) can be handled similarly. The rest is similar to proof of Theorem \ref{thm-main-locationdependent} and hence it follows that both $(\hat{u}_d, 0)$ and $(0,\hat{v}_D)$ are neutrally stable, which is a contradiction.

The proofs of (ii), (iii) and (iv) are omitted since they are similar to that of Theorem \ref{thm-main-locationdependent} (ii), (iii) and (iv) respectively.
\end{proof}

\section*{Acknowledgement}
The authors would like to thank Prof. Yuan Lou for many inspiring discussions  and his constant encouragement and support, without whom this paper would not be possible. The second author also wishes to thank the staff at Department of Mathematics in the Ohio State University for their warm hospitality she received during her visit in the winter
of 2015, where part of this paper was written.

\appendix
\section{Single equations}

This section is about the existence and uniqueness of positive solutions to single equations. In fact, we consider a more general problem as follows:
\begin{equation}\label{single-general}
u_t(x,t) =\mathcal{L}[u] + f(x,u) \doteq d \int_{\Omega}k(x,y)u(y,t)dy +f(x,u),
\end{equation}
where $k(x,y)$ satisfies \textbf{(C2)} and $f(x,u)$ satisfies
\begin{itemize}
\item[\textbf{(f1)}] $f\in C(\bar\Omega\times \mathbb R^+, \mathbb R)$, $f$ is $C^1$ continuous  in $u$ and $f(x,0)=0$;
\item[\textbf{(f2)}] For $u>0$, $f(x,u)/u$ is  strictly decreasing in $u$;
\item[\textbf{(f3)}] There exists $C_1>0$ such that  $d \int_{\Omega}k(x,y)dy +f(x,C_1)/C_1\leq 0$ for all $x\in\Omega$.
\end{itemize}

To study the existence of positive steady state of (\ref{single-general}), it is natural to consider the local stability of the trivial solution $u\equiv 0$, which is determined by the signs of
$$
\lambda_0=\sup \left\{\textrm{Re}\, \lambda\, |\, \lambda\in \sigma(\mathcal{L}+f_u(x,0)   \right\}.
$$
Inspired by the work of Beresticki, Nirenberg and Varadham in \cite{BNV1994}, to study principal eigenvalue of nonlocal operators without assuming the symmetry of dispersal kernels, Coville \cite{Coville2010} was the first to   define
\begin{equation}\label{principaleigenvalue}
\lambda^* =\inf \left\{  \lambda\in \mathbb R \ \big |\ \exists \phi\in C(\bar\Omega), \phi>0 \textrm{ in } \bar\Omega \textrm{ such that }  \mathcal{L}[\phi] +f_u(x,0)\phi - \lambda\phi\leq 0 \right\}
\end{equation}
and it has been proved in \cite{CovilleLiWang} that $\lambda_0=\lambda^*.$

\begin{thm}\label{thm-single}
Under the assumptions \textbf{(C2)}, \textbf{(f1)}, \textbf{(f2)} and \textbf{(f3)}, problem (\ref{single-general}) admits a unique positive steady state in $C(\bar\Omega)$ if and only if $\lambda^*>0$. Moreover, the unique positive steady state, whenever it exists, is globally asymptotically stable, otherwise, $u\equiv 0$ is globally asymptotically stable.
\end{thm}

Thanks to Theorem \ref{thm-single}, the following result follows immediately.
\begin{cor}\label{Cor-single}
Assume that \textbf{(C1)}, \textbf{(C2)} hold.
\begin{itemize}
\item[(i)] (\ref{singled}) admits a unique positive steady state, denoted by $u_d$, in $C(\bar\Omega)$ if and only if
    $$
\lambda_d=\inf \left\{  \lambda\in \mathbb R \ \big |\ \exists \phi\in C(\bar\Omega), \phi>0 \textrm{ in } \bar\Omega \textrm{ such that }  d \mathcal{K}[\phi] + m \phi -\lambda\phi\leq 0 \right\}>0.
$$
\item[(ii)] (\ref{singleD}) admits a unique positive steady state,  denoted by $v_D$, in $C(\bar\Omega)$ if and only if
    $$
\lambda_D=\inf \left\{  \lambda\in \mathbb R \ \big |\ \exists \phi\in C(\bar\Omega), \phi>0 \textrm{ in } \bar\Omega \textrm{ such that }  D \mathcal{P}[\phi] + M \phi -\lambda\phi\leq 0 \right\}>0.
$$
\end{itemize}
\end{cor}

\begin{proof}[Proof of Theorem \ref{thm-single}]
First, assume that (\ref{single-general}) admits a positive steady state in $C(\bar\Omega)$, denoted by $\theta$, and $\lambda^*\leq 0$. Since $\theta(x)$ satisfies
$$
\mathcal{L}[\theta] +f(x,\theta) =0
$$
and $\theta>0$ in $\bar\Omega$, by assumptions \textbf{(f1)} and \textbf{(f2)}, it is easy to check that there exists $\epsilon>0$ such that
$$
\mathcal{L}[\theta] +f_u(x,0)\theta -\epsilon \theta = \theta \left( - {f(x,\theta)\over \theta} +f_u(x,0) -\epsilon\right)>0\ \textrm{ in } \bar\Omega.
$$
Moreover, according to the definition of $\lambda^*$ in (\ref{principaleigenvalue}), there exists $\phi\in C(\bar\Omega), \phi>0 \textrm{ in } \bar\Omega$ and $  \lambda^*<\lambda <\epsilon$ such that
$
\mathcal{L}[\phi] +f_u(x,0)\phi -\lambda\phi\leq 0,
$
which immediately implies that
$$
\mathcal{L}[\phi] +f_u(x,0)\phi - \epsilon \phi < 0.
$$
Then consider $v_{\ell} =\phi -\ell\theta$ and one sees that for any $\ell>0$,
\begin{equation}\label{appendix-v}
\mathcal{L}[v_{\ell}] +f_u(x,0)v_{\ell} - \epsilon v_{\ell} < 0.
\end{equation}
When $\ell$ is sufficiently small, $v_{\ell} =\phi -\ell\theta >0$ in $\bar\Omega$. Hence increasing $\ell$, there exists $\ell_0$ such that $v_{\ell_0} \geq 0$ in $\bar\Omega$ and $v_{\ell_0}$ touches zero at some $x=x_0\in \bar\Omega$. Then clearly
$$
\mathcal{L}[v_{\ell_0}](x_0) +f_u(x_0,0)v_{\ell_0}(x_0) - \epsilon v_{\ell_0}(x_0) \geq 0,
$$
which contradicts (\ref{appendix-v}). Therefore $\lambda^*>0$.

Now assume that $\lambda^*>0$. First, note that assumptions \textbf{(f2)} and \textbf{(f3)} indicate that for all $c\geq C_1$, $u = c$ is an upper solution of (\ref{single-general}) and thus $u(x,t; c)$ is decreasing in $t$, where $u(x,t; u_0)$ denotes  the solution of (\ref{single-general}) with initial data $u_0$.

Next, let us construct suitable lower solutions. Note that there exists $h(x)\in C(\bar\Omega)$ with $\|h- f_u(\cdot,0)\|_{L^{\infty}} < \lambda^*/3$ such that the principal eigenvalue of
$$
\mathcal{L}[\phi] + h(x)\phi -\lambda \phi =0
$$
exists, denoted by $\tilde{\lambda}$, and the corresponding eigenfunction is denoted by $\tilde{\phi}$, where w.l.o.g., $\|\tilde{\phi}\| _{L^{\infty}}=1$. This is a standard result for linear operator with nonlocal dispersals. For nonsymmetric case, see \cite{Coville2010, CovilleLiWang} for example. Also
$$
\tilde{\lambda} =\inf \left\{  \lambda\in \mathbb R \ \big |\ \exists \phi\in C(\bar\Omega), \phi>0 \textrm{ in } \bar\Omega \textrm{ such that }  \mathcal{L}[\phi] + h(x)\phi -\lambda\phi\leq 0 \right\}.
$$
This implies that
$$
| \lambda^* - \tilde{\lambda} | \leq \|h- f_u(\cdot,0)\|_{L^{\infty}} < \lambda^*/3.
$$
Then it is straightforward to verify that due to assumption \textbf{(f1)}
\begin{eqnarray*}
&& \mathcal{L}[\delta \tilde{\phi}] +f(x, \delta\tilde{\phi})\\
&=& -  h(x)\delta\tilde{\phi} - \tilde{\lambda}\delta\tilde{\phi} + {f(x, \delta\tilde{\phi}) \over \delta\tilde{\phi}} \delta\tilde{\phi}\\
&\geq & \delta\tilde{\phi} \left( \lambda^* -\left\|h- f_u(\cdot,0)\right\|_{L^{\infty}}-| \lambda^* - \tilde{\lambda} | -  \left\| {f(\cdot, \delta\tilde{\phi}) \over \delta\tilde{\phi}}  - f_u(\cdot,0) \right\|_{L^{\infty}}  \right)\\
&> & \delta\tilde{\phi} \left( {\lambda^* \over 3} -  \left\| {f(\cdot, \delta\tilde{\phi}) \over \delta\tilde{\phi}}  - f_u(\cdot,0) \right\|_{L^{\infty}}  \right)>0
\end{eqnarray*}
for $\delta>0$ small enough. We conclude that $u = \delta \tilde{\phi}$ is a lower solution of (\ref{single-general})
for any $\delta>0$ sufficiently  small and thus $u(x,t; \delta \tilde{\phi})$ is increasing in $t$.

Based on the upper and lower solutions constructed above, one sees that there exist $\hat{u}\geq \underline{u} >0$, with
$\hat{u},\ \underline{u}\in L^{\infty}$, such that
$$
\lim_{t\rightarrow +\infty} u(x,t; c) = \hat{u}(x)\ \textrm{and} \
\lim_{t\rightarrow +\infty} u(x,t; \delta \tilde{\phi}) = \underline{u}(x)\ \ \textrm{pointwisely}.
$$
Moreover, both $\hat{u}$ and $\underline{u}$ are positive steady states of (\ref{single-general}) in $L^{\infty}$. Then applying the same arguments in \cite[Page 434]{BZh}, one sees that $\hat{u},\ \underline{u}\in C(\bar\Omega)$. Thanks to \cite[Theorem 7.13]{Rudin}, we have
\begin{equation}\label{appendix-converge}
\lim_{t\rightarrow +\infty} u(x,t; c) = \hat{u}(x)\ \textrm{and} \
\lim_{t\rightarrow +\infty} u(x,t; \delta \tilde{\phi}) = \underline{u}(x) \ \textrm{in}\ L^{\infty}(\Omega).
\end{equation}

Due to (\ref{appendix-converge}) and the fact that upper solutions can be arbitrarily large and lower solutions can be arbitrarily small, to show the uniqueness and global convergence of the positive steady state of (\ref{single-general}), it suffices to demonstrate $\hat{u} = \underline{u}$. Since $\hat{u}\geq \underline{u}$, we only need show $\hat{u}\leq \underline{u}$.

Let $\ell_1 = \inf \{ \ell\ | \ \hat{u}\leq \ell \underline{u} \}$.  If $\ell_1>1$, using \textbf{(f2)}, it is easy to check that
$$
\mathcal{L}[\ell_1 \underline{u}] +f(x, \ell_1 \underline{u})= \mathcal{L}[\ell_1 \underline{u}] + {f(x, \ell_1 \underline{u})\over \ell_1 \underline{u}}\ell_1 \underline{u}
< \mathcal{L}[\ell_1 \underline{u}] + {f(x,   \underline{u})\over  \underline{u}}\ell_1 \underline{u} =0,
$$
i.e., $\ell_1 \underline{u}$ is an upper solution.
Note that $\hat{u}\leq \ell_1  \underline{u}$. If $\ell_1  \underline{u}- \hat{u}$ touches zero at some $x=x_0$ in $\bar\Omega$, then
\begin{eqnarray*}
\mathcal{L}[\ell_1 \underline{u}](x_0)+f(x_0, \ell_1 \underline{u}(x_0))= \mathcal{L}[\ell_1 \underline{u}-\hat{u}](x_0)+f(x_0, \ell_1 \underline{u}(x_0))-f(x_0,  \hat{u}(x_0))\geq 0.
\end{eqnarray*}
This is a contradiction. If $\ell_1  \underline{u}- \hat{u}$ never touches zero   in $\bar\Omega$, it contradicts the definition of $\ell_1$. Therefore, $\ell_1\leq 1$ and thus $\hat{u}\leq \underline{u}$.

At the end, when there is no positive steady state, due to \cite[Theorem 7.13]{Rudin},
one sees that for any $c\geq C_1$,
$
\lim_{t\rightarrow +\infty} u(x,t; c) = 0 \ \textrm{in}\ L^{\infty}(\Omega).
$
The proof is complete.
\end{proof}

\section{At most one semi-trivial steady state of (\ref{original})}
\begin{thm}\label{thm-nonsymmetric}
Assume that \textbf{(C1)}, \textbf{(C2)} hold, $b,c>0$. For the global dynamics of the system (\ref{original}) with nonlocal operators $\mathcal{K}$ and $\mathcal{P}$ defined as either type \textbf{(N)} or \textbf{(D)}, the following statements hold:
\begin{itemize}
\item[(i)] $(0,0)$ is locally stable or neutral stable, then $(0,0)$ is globally asymptotically stable;
\item[(ii)] System (\ref{original}) admits only one semi-trivial steady state  $(u_d, 0)$, then  it is globally asymptotically stable;
\item[(iii)] System (\ref{original}) admits only one semi-trivial steady state $(0,v_D)$, , then  it is globally asymptotically stable.
\end{itemize}
\end{thm}

\begin{proof}
Suppose that  $(0,0)$ is locally stable or neutral stable, i.e.,
$$
\sup \left\{\textrm{Re } \lambda \ | \ \lambda\in\sigma(\mathcal{L}_{(0,0)}) \right\} \leq 0,
$$
where
\begin{equation}\label{lin-0}
\mathcal{L}_{(0,0)} {\phi\choose\psi}={d\mathcal{K}[\phi]+ m(x) \phi \choose D\mathcal{P}[\psi]+ M(x) \psi}.
\end{equation}
Note that it is proved in \cite{CovilleLiWang} that
$$
\lambda_d = \sup \left\{\textrm{Re } \lambda \ | \ \lambda\in\sigma(d\mathcal{K} + m(x) ) \right\}
$$
and
$$
\lambda_D = \sup \left\{\textrm{Re } \lambda \ | \ \lambda\in\sigma(D\mathcal{P} + M(x) ) \right\}
$$
where $\lambda_d$ and $\lambda_D$ are defined in Corollary \ref{Cor-single}. Moreover, it is routine to show that
$$
\sup \left\{\textrm{Re } \lambda \ | \ \lambda\in\sigma(\mathcal{L}_{(0,0)}) \right\} = \max \{\lambda_d,  \lambda_D  \}.
$$
Hence $\lambda_d, \lambda_D \leq 0$. This, by Corollary \ref{Cor-single}, implies that neither (\ref{singled}) nor (\ref{singleD}) admits  positive steady states and thus (\ref{original}) has no semi-trivial steady states.

Furthermore, we claim that {\it if (\ref{original}) has a positive steady state, denoted by $(u,v)$, then both (\ref{singled}) and (\ref{singleD}) will admit positive steady states in $C(\bar\Omega)$.} Notice that
$$
d \mathcal{K}[u]  +u(m(x)-u)=  cuv> 0
$$
and
$$
 D \mathcal{P}[v]  +v(M(x)- v) =b uv >0.
$$
Then the claim follows from the arguments employed in the proof of Theorem  \ref{thm-single}.

Therefore, when $(0,0)$ is locally stable or neutral stable, except for $(0,0)$, (\ref{original}) has no other nonnegative steady state. Then based on the fact that (\ref{original}) is a monotone system under the competitive order, it is easy to show that $(0,0)$ is globally asymptotically stable.

Next, suppose that  (\ref{original}) admits only one semi-trivial steady state, w.l.o.g., say $(u_d, 0)$, i.e., (\ref{singled}) has a positive steady state $u_d$, while (\ref{singleD}) admits no positive steady states.
Based on previous arguments, one sees that $(0,0)$ is locally unstable and  (\ref{original}) has no positive steady states. Let $(u(x,t), v(x,t))$ denote the solution of (\ref{original}) with positive initial value $(u_0,v_0)$. Then by the assumption that (\ref{singleD}) admits no positive steady states, it is easy to see that
\begin{equation}\label{pf-v-0}
\lim_{t\rightarrow +\infty} v(x,t) = 0  \ \textrm{in}\ L^{\infty}(\Omega).
\end{equation}
To verify the convergence of $u(x,t)$, first by Corollary \ref{Cor-single}, $u_d$ exists is equivalent to $\lambda_d< 0$. Thus for $\epsilon>0$ small,
$$
w_t = d \mathcal{K}[w]  +w(m(x)-\epsilon -w)
$$
admits a unique positive steady state $w_{\epsilon}$, which is globally asymptotically stable. According to (\ref{pf-v-0}), there exists $T_{\epsilon}$, such that $0<v(x,t)<\epsilon$ for $t>T_{\epsilon}$. Hence one sees that
$$
w(x,t; u(x,T_{\epsilon})) < u(x, t+T_{\epsilon}) < \tilde{u}(x,t; u(x,T_{\epsilon})),
$$
where $\tilde{u}(x,t; u(x,T_{\epsilon}))$ is the solution to the problem
\begin{equation*}
\begin{cases}
\tilde{u}_t= d \mathcal{K}[\tilde{u}]  +\tilde{u} (m(x)-\tilde{u} ) &\textrm{in } \Omega\times[0,\infty),\\
\tilde{u}(x,0)=u(x,T_{\epsilon}),  &\textrm{in } \Omega.
\end{cases}
\end{equation*}
This yields that
$$
w_{\epsilon}\leq \liminf_{t\rightarrow +\infty} u(x,t)\leq \limsup_{t\rightarrow +\infty} u(x,t)\leq u_d.
$$
Due to the expression
$$
w_{\epsilon}(x) ={1\over 2} \left( m(x)-\epsilon + \sqrt{(m(x)-\epsilon)^2-4 d \mathcal{K}[\tilde{w_{\epsilon}}]}\right),
$$
it is standard to show that $w_{\epsilon}\rightarrow u_d$ in $L^{\infty}(\Omega)$ as $\epsilon \rightarrow  0$. Hence
$$
\lim_{t\rightarrow +\infty} u(x,t) = u_d \ \textrm{in}\ L^{\infty}(\Omega).
$$
Therefore, $(u_d, 0)$ is globally asymptotically stable.
\end{proof}

\end{document}